\documentclass[review,onefignum,onetabnum]{siamart171218}

\usepackage{mathtools}
\usepackage{amsmath,amssymb,amstext}   
\usepackage{tikz, tikz-3dplot}
\usepackage{url}
\usepackage{xspace,enumerate}
\usepackage{datetime}

\definecolor{cof}{RGB}{219,144,71}
\definecolor{pur}{RGB}{186,146,162}
\definecolor{greeo}{RGB}{91,173,69}
\definecolor{greet}{RGB}{52,111,72}

\newcommand{\Z}{\mathbb Z}

\newcommand{\K}{\mathbb{K}}
\newcommand{\F}{\mathbb{F}}

\newcommand{\cclassmacro}[1]{\texorpdfstring{\textbf{#1}}{#1}\xspace}
\newcommand{\p}{\cclassmacro{P}}
\newcommand{\np}{\cclassmacro{NP}}

\newcommand{\cB}{\mathcal B}

\newcommand{\cI}{\mathcal I}

\newcommand{\cV}{\mathcal V}

\newcommand{\aalpha}{\boldsymbol{\alpha}}
\newcommand{\bbeta}{\boldsymbol{\beta}}
\newcommand{\eeta}{\boldsymbol{\eta}}

\DeclareMathOperator{\supp}{supp}

\newsiamremark{example}{Example}
\newsiamremark{openproblem}{Problem}
\crefname{example}{Example}{Examples}
\crefname{openproblem}{Problem}{Problems}

\usepackage{lipsum}
\usepackage{amsfonts}
\usepackage{graphicx}
\usepackage{epstopdf}
\usepackage{algorithmic}
\ifpdf
  \DeclareGraphicsExtensions{.eps,.pdf,.png,.jpg}
\else
  \DeclareGraphicsExtensions{.eps}
\fi

\newsiamremark{remark}{Remark}
\newsiamremark{hypothesis}{Hypothesis}
\crefname{hypothesis}{Hypothesis}{Hypotheses}
\newsiamthm{claim}{Claim}

\headers{Graphs with large girth and the Nullstellensatz}{J. Romero, L. Tun\c{c}el}

\title{Graphs with large girth and chromatic number are hard for Nullstellensatz \thanks{Preliminary versions of some of the results appeared in the first author's PhD Thesis \cite{julian}.} %\thanks{Submitted to the editors \today}. 
\funding{This work was funded in part by NSERC Discovery Grants, Tutte Scholarship, U.S. Office of Naval Research under award numbers: N00014-15-1-2171 and N00014-18-1-2078. This financial support is gratefully acknowledged.} }

\author{Julian Romero\thanks{Amazon.com, Inc., Vancouver, BC V6B 1X4, Canada
  (\email{jromerob@amazon.com}, \email{jaromero@uwaterloo.ca})}
\and Levent Tun\c{c}el \thanks{Department of Combinatorics and Optimization, Faculty of Mathematics, University of Waterloo, Waterloo, ON N2L 3G1, Canada
  (\email{levent.tuncel@uwaterloo.edu}).}
}

\usepackage{amsopn}

\makeatletter
\newcommand*{\addFileDependency}[1]{% argument=file name and extension
  \typeout{(#1)}% latexmk will find this if $recorder=0 (however, in that case, it will ignore #1 if it is a .aux or .pdf file etc and it exists! if it doesn't exist, it will appear in the list of dependents regardless)
  \@addtofilelist{#1}% if you want it to appear in \listfiles, not really necessary and latexmk doesn't use this
  \IfFileExists{#1}{}{\typeout{No file #1.}}% latexmk will find this message if #1 doesn't exist (yet)
}
\makeatother

\newcommand*{\myexternaldocument}[1]{%
    \externaldocument{#1}%
    \addFileDependency{#1.tex}%
    \addFileDependency{#1.aux}%
}
\ifpdf
\hypersetup{
  pdftitle={Graphs with large girth and chromatic number are hard for Nullstellensatz},
  pdfauthor={J. Romero, L. Tun\c{c}el}
}
\fi

\myexternaldocument{ex_supplement}

\begin{document}

\maketitle

% REQUIRED
\begin{abstract}
We study the computational efficiency of approaches, based on Hilbert's Nullstellensatz, which use systems of linear equations for detecting non-colorability of graphs having large girth and chromatic number. We show that for every non-$k$-colorable graph with $n$ vertices and girth $g>4k$, the algorithm is required to solve systems of size at least $n^{\Omega(g)}$ in order to detect its non-$k$-colorability.
\end{abstract}

% REQUIRED

% REQUIRED

%Introduction
\section{Introduction}

The use of algebraic geometry methods in combinatorial optimization has increased in popularity over the last couple of decades. The ability to reformulate hard optimization problems using simple multivariate polynomial formulations can be quite appealing. Especially, because these formulations can often be \textit{relaxed} to obtain computationally tractable approximations to a hard combinatorial optimization problem. Popular examples of this are hierarchy  based approaches such as the Balas-Ceria-Cornu\'{e}jols, Sherali-Adams, Lov\'asz-Schrijver and Sum of Squares, or Lasserre relaxations for optimization problems (see for instance \cite{AuT2016,CT12}), where \textit{Linear Programming} and \textit{Semidefinite Programming} problems are used as building blocks to construct tighter and tighter relaxations to the optimization problem.  

Another algebraic approach for combinatorial problems, implicit in the work of Beame \textit{et al.} \cite{Beame94} and later proposed by De Loera \textit{et al.} \cite{loeraetal08} and Margulies \cite{Margulies08}, uses Hilbert's Nullstellensatz to create a hierarchy of relaxations based on \textit{Systems of Linear Equations}.  More concretely, one first formulates (the decision version of) a combinatorial problem using a system of  polynomial equations
\begin{equation}\label{Eq:PolySys0}
f_1(x)=f_2(x)=\cdots=f_m(x)=0,
\end{equation}
for polynomials $f_i\in \K[x_1, \dots, x_n]$ on $n\geq 1$ variables over some field $\K$. This is done in a way that the system \eqref{Eq:PolySys0} has a solution over the algebraic closure $\overline{\K}$ if and only if the combinatorial problem has a solution. By Hilbert's Nullstellensatz (see for instance \cite{cox}), the problem does not have a solution if and only if there exist polynomials $r_1, \dots, r_m\in \K[x_1,\dots, x_n]$, which we call \textit{Nullstellensatz Certificates}, such that 
\begin{equation}\label{Eq:NulCert0}
r_1(x)f_1(x)+\dots+r_m(x)f_{m}(x)=1.
\end{equation}
Although the Nullstellensatz has been used to obtain interesting results in combinatorics (see \cite{Alon1999}),  a crucial observation made by De Loera et al. is that the existence Nullstellensatz Certificates $r_1, \dots, r_m$ of degree at most $d$,  can be determined using a \textbf{system of linear equations} over $\K$. This allows us to use a hierarchy of systems of linear equations to solve the combinatorial problem. For general systems of polynomial equations, the maximum degree of the Nullstellensatz Certificates can be exponential in $n$ and $m$. However, for systems based on combinatorial problems with a special structure, the maximum degree of the certificates can be small, which makes the method potentially computationally attractive, as  solving systems of linear equations is usually significantly faster than solving linear or semidefinite programs of comparable size.  

In a series of papers, De Loera \textit{et al.} \cite{loeraetal08,DeLoera10, DeleoraLeeMarguliesMiller2015} studied the Nullstellensatz approach for several combinatorial problems, with special attention given to \textit{Graph Coloring}. Recall that  for a graph $G=(V,E)$ and an integer $k\geq 2$, the graph $G$ is \textit{$k$-colorable} if it is possible to assign $k$ colors to  its vertices in a way that no pair of adjacent vertices have the same color.  The polynomial formulation they used for the $k$-coloring problem, due to  Bayer \cite{Bayer82}, is given by the system

\begin{equation}\label{BCOL_k}
\tag{BCOL}
\begin{aligned}
p_u(x)&:=x_u^k-1=0, &\quad \forall u\in V,\\
q_{uv}(x)&:=\frac{x_u^{k}-x_v^{k}}{x_u-x_v}=0, &\quad \forall \{u,v\}\in E.
\end{aligned}
\end{equation}

The computational experiments of De Loera \textit{et al.} showed the potential applications of the Nullstellensatz method in detecting non-$3$-colorability of graphs. They were able to solve known hard instances, such as the Mizuno and Nishihara \cite{hard3} families, of non-$3$-colorable graphs with up to a thousand of vertices over $\F_2$. In fact, no graph needing a Nullstellensatz certificate of degree larger than \textit{four} was encountered at the time. This was quite surprising, given the fact that unless $\p=\np$, families of  graphs needing large Nullstellesatz Certificates must exist.   

The problem of finding non-$k$-colorable graphs needing large certificates was settled until recently by Lauria and Nordstr\"om \cite{LN17}. Their proof consisted of a nifty \textit{Polynomial Calculus} (PC) reduction from the Functional Pigeonhole Principle (FPHP) to Graph Coloring. Since lower bounds for the degrees of PC proofs for special instances of FPHP had already been shown in \cite{MN15}, the result follows for $k$-colorability as well, implying the need of large Nullstellensatz Certificates for those instances. The graphs found in \cite{LN17} yield asymptotically tight results in the sense that these graphs need certificates whose degrees share the same order of magnitude as the number vertices in the graph times $k$. 

Despite these results, it is still wide open  to determine the intrinsic \textit{combinatorial properties} of non-$k$-colorable graphs which require them to have small (or large) Nullstellensatz Certificates. For instance, De Loera et al. \cite{DeLoera10} fully characterized all non-$3$-colorable graphs having a Nullstellensatz certificate of degree one over $\F_2$.  However, it is still an open question to characterize all non-$3$-colorable graphs needing certificate of degree at most four over $\F_2$. The only related result we know is due to Li, Lowenstein and Omar \cite{Omar15}, who showed that no $4$-critical graph with at most $12$ vertices has a Nullstellensatz Certificate of degree larger than four over $\F_2$.

In order to understand this problem further, it is natural to study families of graphs that are hard for coloring problems, such as the instances studied in \cite{hard3} among others. One novel example of these are graphs with large \textit{girth}, the length of the shortest cycle in the graph. A classical result of Erd\"{o}s \cite{erdosgirth} establishes the existence of graphs having arbitrarily large girth and chromatic number. These graphs locally look like trees. Thus, it is "easy" to color them locally, however a global understanding of the graph is needed in order to determine their chromatic number. 

Explicit examples of graphs with large girth and chromatic number can be found in \cite{lovaszgirth, ramanujan} and more recently in \cite{Alon1}. Most of the known explicit examples of $k$-colorable graphs having large girth are also fairly large in size. In fact, if we denote by $n(g,k)$ the number vertices of the smallest graph having chromatic number $k$ and girth $g$, it is known that (see \cite{Exoo} for lower bound and \cite{Marshall} for upper bound)
\begin{equation}
\frac{2(k-2)^{(g-1)/2}-2}{k-3}\leq n(g,k)\leq 9gk^{6g+1}.
\end{equation}
In particular, the size of these graphs are exponential in their girth. In this article, we show how to exploit the local structure of this family of graphs to prove that non-$k$-colorable graphs with large girth also need large Nullstellensatz Certificates. More concretely, we show the following. 
 
\begin{theorem}\label{Thm:Main}
Let $G=(V,E)$ be a graph with chromatic number $\chi(G)=k+1$ and girth $g>4k$. Then, for every 	non-negative integer $d$ satisfying
\begin{equation}\label{Eq:MainTheorem}
d+k-1<\frac{g}{4k},
\end{equation}
$G$ has no Nullstellensatz Certificates of degree at most $d$ for the system \eqref{BCOL_k}.
\end{theorem}

We follow a similar approach to  the work of Razborov \cite{Raz98}, Aleknovich and Razborov \cite{AR03} (later expanded in \cite{MN15}) for boolean systems. A key idea is to understand the principal ideal of subsystems of \eqref{BCOL_k} corresponding to local subgraphs. For graphs with large girth, small enough subgraphs are forests, whence $k$-colorable for $k\geq 2$. In particular, no Nullstellensatz Certificate exists for their corresponding subsystems. This can be witnessed by what we call a \textit{Dual Nullstellensatz Certificate} of the subsystem, which is constructed using information from the \textit{standard monomials} of the ideal generated by the polynomials in the subsystem.  These "local" dual certificates are then patched to create a "global" dual certificate for the whole system, thus proving that the system does not admit a Nullstellensatz Certificate of certain degree. 

We should point out that, while being partly inspired by the work in \cite{AR03} and \cite{MN15},  our work does not seem to be a direct consequence of these studies. Besides them being applicable to Boolean systems only,  one key component in their result requires the polynomial-variable incidence graph (or a clustering of it, as shown in \cite{MN15}) of the system of polynomial equations to be a "good enough" expander. However,  large girth alone does not seem to imply good expansion properties of the polynomial-variable incidence graph (or a clustering of this) for Bayer's polynomial system. In any case, if these properties were to carry over, we would still need to adjust the proofs in \cite{MN15} to apply for more general set of systems of polynomial equations or use a different encoding for the graph coloring problem using Boolean systems, such as the one used in \cite{LN17}.

Instead,  we found that understanding  the structure of the principal ideal of local subsystems, as it is done in \cite{Raz98} for systems arising from the \textit{Pigeonhole Principle}, was critical and it allowed us to further understand the behavior of the Nullstellensatz and Polynomial Calculus proof systems for graph coloring in general.

%Preliminaries
\section{Notation and Preliminaries.}

For a positive integer $n$, let $[n]$ be the set $\{1,\dots, n\}$. Let $\K$ be an algebraically closed field and let $\K[x_1,\dots, x_n]$ be the ring of polynomials with coefficients in $\K$. Monomials in $\K[x_1,\dots, x_n]$ are denoted using multi-index notation $x^{\aalpha}:=x_{1}^{\aalpha_1}\cdots x_{n}^{\aalpha_n}$ where $\aalpha:=(\aalpha_1,\dots, \aalpha_n)$ is a non-negative integer vector. The degree of the monomial $x^{\aalpha}$ is defined as
$$
|\aalpha|:=\aalpha_1+\dots+\aalpha_n
$$
and the support of $\aalpha$, denoted by $\supp(\aalpha)$, is the set of indices $i\in [n]$ such that $\aalpha_i>0$. For a non-negative integer $d$, the set $\K[x_1,\dots, x_n]_{\leq d}$ denotes the vector space of polynomials of degree at most $d$, i.e., the set of polynomials $f\in \K[x_1,\dots, x_n]$ that can be written as
\begin{equation}\label{Eq:PolyDef}
f(x)=\sum_{|\aalpha|\leq d}f_{\aalpha}x^{\aalpha}
\end{equation}
for some scalars $f_{\aalpha}\in \K$. The support of $f$, denoted by $\supp(f)$ is the set of multi-indices $\aalpha$ such that $f_{\aalpha}\neq 0$.

Division algorithms over $\K[x_1,\dots, x_n]$ are possible once a \textit{monomial ordering} has been established. Some commonly used monomial orderings are
\begin{enumerate}
\item The \textit{Lexicographic Order} (LEX). For any pair of monomials $x^{\aalpha}$ and $x^{\bbeta}$ we write $x^{\aalpha}\preceq_{LEX} x^{\bbeta}$ if there exists a positive integer $i\in [n]$ such that $\aalpha_j=\bbeta_j$ for all $j<i$ and $\aalpha_i<\bbeta_i$. 
\item The \textit{Graded Lexicographic Order} (GLEX). For any pair of monomials $x^{\aalpha}$ and $x^{\bbeta}$ we write $x^{\aalpha}\preceq_{GLEX} x^{\bbeta}$ if either $|\aalpha|<|\bbeta|$, or $|\aalpha|=|\bbeta|$ and $x^{\aalpha}\preceq_{LEX} x^{\bbeta}$.
\end{enumerate}
Unless stated otherwise, in this article we will use the graded lexicographic order (GLEX) and we set $\preceq:=\preceq_{GLEX}$. For a given polynomial $f\in \K[x_1,\dots, x_n]$,  its \textit{leading monomial}, denoted by $LM(f)$, is the largest (in GLEX order) monomial in the support of $f$.  
 
Given a finite set of polynomials $F:=\{f_1,\dots, f_m\}\subseteq \K[x_1,\dots, x_n]$, the \textit{ideal} generated by $F$ is the set  
\begin{equation}
\langle F\rangle:=\{r_1f_1+\dots+r_mf_{m}: r_i\in \K[x_1,\dots, x_n]\}.
\end{equation}
The \textit{variety} defined by $F$, denoted by $\cV(F)$, is the set of solutions to the system
\begin{equation}\label{Eq:PolySys}
f_1(x)=f_2(x)=\cdots=f_m(x)=0.
\end{equation}
The following notion will be used quite often in this article. 
\begin{definition}
Let $\cI$ of $\K[x_1,\dots, x_n]$ be an ideal. We denote by $LM(\cI)$ the set of all leading monomials of polynomials in $\cI$. If $x^{\aalpha}\in LM(\cI)$, then we say that the monomial $x^{\aalpha}$ is \textbf{reducible modulo} the ideal $\cI$, otherwise we say that the monomial is \textbf{irreducible modulo} $\cI$.
\end{definition}
Recall that a \textit{Gr\"{o}bner basis} of $\cI$ is a set of polynomials $\{g_1,\dots, g_r\}$ such that $\cI=\langle g_1,\dots, g_r\rangle$ and 
$$
\langle LM(\cI)\rangle=\langle LM(g_1), \dots, LM(g_r)\rangle.
$$
In particular, a monomial $x^{\aalpha}$ is reducible modulo the ideal $\cI$ if and only if $x^{\aalpha}$ is divisible by $LM(g_i)$ for some $i\in [r]$. The following notion is also key in this article. 
\begin{definition}
Let $\cI$ be an ideal and let $f\in \K[x_1,\dots, x_n]$ be any polynomial. Let $g_1,\dots, g_r$ be a Gr\"{o}bner Basis of $\cI$. The \textit{reduction} (or, \textit{normal form}) of $f$ modulo $\cI$ is the remainder of the division of $f$ by the basis $g_1,\dots, g_r$, i.e., it is the unique polynomial $\phi_{\cI}(f)\in \K[x_1,\dots, x_n]$ such that
\begin{enumerate}
\item $f=g+\phi_{\cI}(f)$ for some $g\in \cI$, and 
\item no monomial in $\phi_{\cI}(f)$ is reducible modulo $\cI$.
\end{enumerate}   
\end{definition}

Hilbert's Nullstellensatz \cite{Hil1893, TaoNulls, cox}, in its most basic form, is the statement that $\cV(F)=\emptyset$ if and only if $1\in \langle F\rangle$. In other words, the system \eqref{Eq:PolySys} has no solution if and only if
\begin{equation}\label{Eq:NulCert}
r_1(x)f_1(x)+\dots+r_m(x)f_{m}(x)=1
\end{equation}
for some polynomials $r_1,\dots, r_m\in \K[x_1,\dots, x_n]$. In particular, we say that the polynomials $r_1, \dots, r_m$ in \eqref{Eq:NulCert} form a \textit{Nullstellensatz Certificate} of degree $d$ if each polynomial $r_i$ has degree at most $d$. A crucial observation is that Nullstellensatz Certificates of degree $d$ can be found by solving a system of linear equations as easily seen from \eqref{Eq:NulCert}. Indeed,  such system is found by considering each $r_i$ in \eqref{Eq:NulCert} with variable coefficients and then equating the resulting polynomial with the constant polynomial $1$. Let us denote this system of equations by
\begin{equation}\label{Eq:LinSystem}
A_{F,d} y=b_{F,d}.
\end{equation}
It is a result of Kollar \cite{kollar} that if the system \eqref{Eq:PolySys} is infeasible, then there exists a Nullstellensatz Certificate of degree $d_K:=\max(3,d_{\max})^{\min(n,m)}$ where $d_{\max}$ is the largest degree of the polynomials in $F$. However, in order for the algorithm to detect feasibility, a system of size  
$$\Theta\left(\frac{4^{d_K}}{d_k^{1/2}}\right)$$ should be solved. Since $d_k$ is exponential in $n$, the system ends up being doubly exponential in the number of variables. There are some cases where such size can be reduced. For instance, if the polynomials in $F$ have no common root at infinity, a theorem of Lazard \cite{Lazard77} allows us to reduce Kollar's bound to $d_{L}:=n\cdot(d_{\max}-1)$. Thus, feasibility can be checked by solving a linear system of size singly exponential in the number of variables.

\subsection{Graph Coloring and Subgraph Ideals}
Let $G=(V,E)$ be a graph with vertex set $V=[n]$ for some $n\in \Z_+$. All graphs in this article are simple, finite and undirected. A graph is said to be $k$-colorable if there exists a map $\kappa:V\rightarrow [k]$ such that $\kappa(u)\neq \kappa(v)$ for all edges $uv\in E$. The minimum $k\in \Z_+$ for which $G$ is $k$-colorable is called the \textit{chromatic number} of $G$ and we denote this number by $\chi(G)$. Recall, the \textit{girth} of $G$ is the length of the shortest cycle in $G$.

%The fact that one can cast a $k$-colorability problem as a system of polynomial equations was noted by Bayer \cite{Bayer82} and it is described as follows. Let $\K$ be an algebraically closed field with characteristic relatively prime to $k$, then the system of polynomial equations in $\K[x_u: u\in V]$ given by
%
%\begin{equation}\label{BCOL_k}
%\tag{$\text{BCOL}_k$}
%\begin{aligned}
%p_u(x)&:=x_u^k-1=0, &\quad \forall u\in V,\\
%q_{uv}(x)&:=\frac{x_u^{k}-x_v^{k}}{x_u-x_v}=0, &\quad \forall \{u,v\}\in E,
%\end{aligned}
%\end{equation}
%has a solution if and only if $G$ is $k$-colorable. Indeed, the first set of polynomials assigns a $k$-th root of the unity to each variable $x_u$ with $u\in V$, and the second set of polynomials guarantees that $x_u\neq x_v$ for all $\{u,v\}\in E$. This last condition is easily seen from the equation
%$$
%0=p_{u}(x)-p_v(x)=( x_u- x_v)\cdot q_{uv}( x).
%$$
Let $\K$ be a field of characteristic not dividing $k$. For a vertex $u\in V$ and edge $vw\in E$, let $p_u(x)$ and $q_{vw}(x)$ be the polynomials defined in Bayer's formulation:
\begin{equation}
\tag{BCOL}
\begin{aligned}
p_u(x)&:=x_u^k-1=0, &\quad \forall u\in V,\\
q_{uv}(x)&:=\frac{x_u^{k}-x_v^{k}}{x_u-x_v}=0, &\quad \forall \{u,v\}\in E.
\end{aligned}
\end{equation}
Notice that the graph $G$ is $k$-colorable if and only if \eqref{BCOL_k} has a solution: the first set of polynomial equations tells us that we aim to color the graph with $k$-roots of the unity and the second set of polynomial equations encode the fact that no pair of adjacent vertices can be assigned the same $k$-th root of the unity. Indeed, if $x_u=x_v=\zeta$ for some root of the unity $\zeta\in \K$ then
\begin{equation}
\begin{aligned}
q_{uv}(x)|_{x_u=x_v=\zeta}&=x_u^{k-1}+x_u^{k-2}x_v+\dots+x_ux_v^{k-2}+x_v^{k-1}|_{x_u=x_v=\zeta},\\
&=k\cdot \zeta^{k-1}\neq 0,
\end{aligned}
\end{equation}
as the characteristic of $\K$ does not divide $k$. Let $\cI_{V,k}$ be the ideal generated by the polynomials $p_u$ with $u\in V$. Consider the quotient ring $R_{V,k}:=\K[x_u: u\in V]/\cI_{V,k}$, i.e., the set of all congruent classes of polynomials modulo the ideal $\cI_{V,k}$. Then, every polynomial $f$ is congruent with the polynomial 
$$
f(x)\equiv \sum_{\aalpha \in \Z_k^V} c_{\aalpha} x^{\aalpha} \quad \mod \cI_{V,k},
$$
for some $c_{\aalpha}\in \K$ and  $\aalpha\in \Z_k^V$, where $\Z_k$ is the set of integers modulo $k$. Such a representation exhibits the fact that $R_{V,k}$ is a finite dimensional vector space, which is isomorphic to the space of functions $\aalpha\mapsto c_{\aalpha}$ mapping $\Z_k^V$ to $\K$. 

Given a subset of edges $F\subseteq E$, we let $\cI_{F}$ be the ideal of $R_{V,k}$ generated by the polynomials $q_{uv}$ with $uv\in F$. Since the polynomials $x^k_u-1$ are square free, the ideal $\cI_{F}$ is radical (see \cite[Proposition 2.7]{cox2}). Thus, we have that $f\in \cI_{F}$ if and only if $f(a)=0$ for every valid $k$-coloring $a=(a_w)_{w\in V(F)}$ of the graph induced by $F$, that is $a_w^k=1$ for all $w\in V(F)$ and $a_v\neq a_w$ for all $vw\in F$.

Clearly, the existence of a Nullstellensatz Certificate of degree $d$ for \eqref{BCOL_k} guarantees the existence of polynomials $r_{uv}$ for $uv\in E$ of degree at most $d$ such that
\begin{equation}\label{Eq:NulMODIkd}
\sum_{uv\in E}r_{uv}(x) q_{uv} (x)\equiv 1 \mod \cI_{V,k}.
\end{equation}
Therefore, in order to find lower bounds for Nullstellensatz Certificates, it is enough to show that there are no polynomials $r_{uv}$ of degree at most $d$ in $R_{V,k}$ satisfying \eqref{Eq:NulMODIkd}. This in turn can be certified using the following lemma. In the lemma, $e_u\in \Z_{k}^V$ with $u\in V$ denote the standard vectors of $\Z_k^V$.
\begin{lemma}\label{Lem:DCOL}
Let $G=(V,E)$ be a graph. Suppose there exists some vector $\lambda=(\lambda_{\aalpha})_{\aalpha\in \Z_k^V, |\aalpha|\leq d+k-1}$ with entries in $\K$ such that
\begin{equation}\tag{DCOL}\label{DCOL}
\begin{aligned}
\textstyle\sum_{r\in \Z_k} \lambda_{\aalpha+r(e_u-e_v)-e_v} &=0, & \forall \aalpha\in \Z_{k}^{V}, \ |\aalpha|\leq d, uv\in E, \\
\lambda_0&=1. & & \\
\end{aligned}
\end{equation}
Then, \eqref{BCOL_k} does not have a Nullstellensatz Certificate of degree $d$, or lower. 
\end{lemma}
\begin{proof}
Consider the linear subspace $N(E,d)\subseteq R_{V,k}$ of all polynomials which can be written in the form 
$$
\sum_{uv\in E} r_{uv} q_{uv}
$$
for some $r_{uv}\in R_{V,k}$ of degree at most $d$. Clearly, $N(E,d)$ is spanned by all polynomials of the form 
$$
x^{\aalpha}q_{uv}(x)=\sum_{r=0}^{k-1}x^{\aalpha}x_{u}^{k-1-r}x_{v}^r, \quad \forall |\aalpha|\leq d, \ \forall \{u,v\}\in E.
$$
Then, $1\notin N(E,d)$ if and only if there exists a linear functional $\lambda:R_{V,k}\rightarrow \K$ such that $\lambda(1)=1$ and $\lambda(f)=0$ for every $f\in N(E,d)$. This last equation is equivalent to the equations
$$
\lambda(x^{\aalpha}q_{uv}(x))=\sum_{r=0}^{k-1}\lambda(x^{\aalpha}x_{u}^{k-1-r}x_{v}^r)=0 \quad \forall \aalpha\in \Z_k^V, |\aalpha|\leq d, \ \forall \{u,v\}\in E.
$$
Using the fact that $R_{V,k}$ is spanned by monomials $x^{\aalpha}$ with $\aalpha\in \Z_k^V$ and setting $\lambda_{\aalpha}:=\lambda(x^{\aalpha})$, we can further rewrite the above equation as 
\begin{equation*}
\begin{aligned}
\textstyle\sum_{r\in \Z_k} \lambda_{\aalpha+r(e_u-e_v)-e_v} &=0, & \forall \aalpha\in \Z_{k}^{V}, \ |\aalpha|\leq d, \{u,v\}\in E. \\
\end{aligned}
\end{equation*}
The statement follows.
\end{proof}
\begin{definition}\label{Def:DNCert}
In matrix notation, let $\hat{A}_{E,d} \lambda= \hat{c}_{E,d}$ be the system \eqref{DCOL}.  Any solution $\lambda$ to this system is called a \textbf{Dual Nullstellensatz Certificate} of degree $d$. 
\end{definition}
\begin{remark}
Notice that the columns of $\hat{A}_{E,d}$ are indexed by monomials of degree at most $d+k-1$. Hence, we can view each row of $\hat{A}_{E,d}$ as a polynomial in $R_{V,k}$.
\end{remark}

%Polynomial Calculus
\section{Polynomial Calculus and Related Work.}
%\begin{definition}
% Let $f_1, \dots, f_m, g \in \K[x_1, \dots, x_n]$ be given polynomials. A \textbf{polynomial calculus (PC) proof} of $g$ from $f_1, \dots, f_n$ is a sequence of polynomials $p_1, \dots, p_t$ such that $p_t=g$ and every $p_i$ in the sequence is either
%\begin{enumerate}
%\item one of the $f_1, \dots, f_m$, or
%\item the linear combination $\alpha p+\beta q$ of some previous polynomials $p,q$ in the sequence, where $\alpha, \beta$ are constants in $\K$, or 
%\item the product $x_j \cdot p$ of one previous polynomial in the sequence $p$ and any variable $x_j$ with $j\in[n]$.
%\end{enumerate} 
%The \textbf{degree} of a PC proof is the maximum degree of the polynomials appearing in the sequence $p_1, \dots, p_t$. When $g=1$, we call such proof a \textbf{polynomial calculus refutation} of $f_1, \dots, f_m$.
% \end{definition} 

A study of lower bounds for Nullstellensatz certificates has already appeared in the context of \textit{Propositional Proof Systems} in a paper by Beame \textit{et al.} \cite{Beame94}. Lower bounds were found for systems of polynomial equations derived from the  \textit{Modular Counting Principle} and the \textit{Pigeonhole Principle} \cite{Beame98}. Later, Clegg \textit{et al.} \cite{Clegg96} worked with a stronger proof system that is now called \textit{Polynomial Calculus} (PC).  In this system, given polynomials $f_1, \dots, f_m\in \K[x_1, \dots, x_n]$, the goal is to find a proof of the statement $1\in \langle f_1, \dots, f_m\rangle$ using a sequence of polynomials 
$p_1, \dots, p_t$, such that $p_t=1$ and every $p_i$ in the sequence is either
\begin{enumerate}
\item one of the $f_1, \dots, f_m$, or
\item the linear combination $\alpha p+\beta q$ of some previous polynomials $p,q$ in the sequence and $\alpha, \beta\in \K$, or
\item the product $x_j \cdot p$ of one previous polynomial $p$ in the sequence and any variable $x_j$ with $j\in[n]$.
\end{enumerate} 

It is not hard to see that the largest degree of the polynomials in the sequence is always smaller than or equal to the degree of a Nullstellensatz certificate (times the maximum of the degrees of the polynomials $f_i$). Thus, lower bounds for the degrees of PC refutations are lower bounds for the degrees of Nullstellensatz certificates. The converse it is not necessarily true. In fact, Clegg \textit{et al.} \cite{Clegg96} showed an exponential separation between these two for some systems of polynomial equations derived from the \textit{House-sitting Principle}, a generalization of the pigeonhole principle.

Notice that lower bounds for the degrees of PC refutations can be found by constructing an operator $\phi:\K[x_1,\dots, x_n]_d\rightarrow \K[x_1,\dots, x_n]_d$ such that 
\begin{enumerate}
\item[a.] $\phi(1)=1$,
\item[b.] $\phi(f_i)=0$ for all $i\in[m]$, and
\item[c.] for every $x^{\aalpha}$ of degree less than $d$ and every $j\in [n]$ \vspace{-2mm}
\[
\phi(x_jx^{\aalpha})=\phi(x_j\cdot\phi(x^{\aalpha})).
\]
\end{enumerate}

Such $\phi$ implies that no PC refutation of degree $d$ exists (see \cite{Raz98}). This is precisely how the bounds in \cite{Raz98}, \cite{AR03}, \cite{MN15} and ultimately the bounds in the present article are built. The idea is to define the operator  $\phi$ in a local fashion: each monomial $x^{\aalpha}$ is assigned a subset $F_{\aalpha}$ of the polynomials $f_1,\dots, f_m$ and then $\phi(x^{\aalpha})$ is defined to be the reduction of $x^{\aalpha}$ modulo the ideal $\langle F_{\aalpha}\rangle$. Clearly,  since we want $\phi$ to satisfy the properties a. to c. above, the sets $F_{\aalpha}$ should be chosen carefully. 

For instance, for Boolean systems, Aleknovich and Razborov \cite{AR03} construct the sets $F_{\aalpha}$ using expandability properties of the \textit{polynomial-variable incidence graph}.  Mik\v{s}a and Nordstr\"om \cite{MN15} use a similar construction using a suitable clustering of the polynomial-variable incidence graph. In our setup, the sets $F_{\aalpha}$ will correspond to suitable sub-forests that we construct using the non-$k$-colorability and large girth of our graphs (what we call \textit{essential graphs} in the section below).

%We should point out the authors in \cite{LN17} were mainly interested in proving the \textit{existence} of a family (efficiently constructable) non-$k$-colorable needing large Nullstellensatz certificates. However, the main goal of this manuscript is to understand what \textit{combinatorial properties} are required for the need of large Nullstellensatz certificates.
%
%In words of Aleknovich and Razborov \cite{AR03}, informally, "\textit{everthing we can infer in small degree we can also infer locally}". This is precisely what motivates our work: if a non-$k$-colorable graph $G$ has a small Nullsetellensatz certificate, then one should be able to detect its non-$k$-colorability by looking at small enough subgraphs of $G$.

%Essential Graphs
\section{Large Girth and Nullstellensatz Certificates}
Throughout this section $G=(V,E)$ will denote a graph with chromatic number $\chi(G)=k+1$ for some $k\geq 3$ and girth $g\geq 3$.  As before, let us identify the vertex set of $G$ with the set $[n]$ and let us consider the ring $R_{V,k}=\K[x_1,\dots, x_n]/\cI_{V,k}$. 

Our goal is to show that $G$ does not have a Nullstellensatz certificate of degree $d\ll g$. As stated in Lemma \ref{Lem:DCOL} and Definition \ref{Def:DNCert}, this can be achieved by finding a Dual Nullstellensatz Certificate $\lambda=(\lambda_{\aalpha})_{|\aalpha|\leq d+k-1}$ satisfying the system
\begin{equation}\label{DCOL_Mat}
\hat{A}_{E,d}\cdot \lambda=\hat{c}_{E,d}.
\end{equation}
The goal of this section is to exploit the sparsity of $G$ to show that such $\lambda$ can be defined locally. More concretely, for each $\aalpha\in \Z_k^V$ of degree at most $d+k-1$ we will associate a subgraph $H_{\aalpha}=(U_{\aalpha},F_{\aalpha})$ of $G$, which we call the \textit{essential graph} of $\aalpha$ (see Definition \ref{Def:EssentialGraphs}). As we will see, the graphs $H_{\aalpha}$ encode the local reducibility of $x^{\aalpha}$, i.e., $x^{\aalpha}$ will be reducible modulo $\cI_{F_{\aalpha}}$ if and only if it is reducible modulo $\cI_{F}$ for every $F\supseteq F_{\aalpha}$ not "too large". Moreover, the sparsity of $G$ will guarantee that each $H_{\aalpha}$ is $k$-colorable, so that the sub-system of linear equations
\begin{equation}\label{DCOL_Mat_Sub}
\hat{A}_{F_{\aalpha},d}\cdot \mu =\hat{c}_{F_{\aalpha}, d}
\end{equation}
has a "local" solution $\mu^{(\aalpha)}=(\mu^{(\aalpha)}_{\bbeta})_{|\bbeta|\leq d+k-1}$ for each $\aalpha$. Finally, by the reducibility property of $H_{\aalpha}$, we will see that the local solutions $\mu^{(\aalpha)}$ for each $\aalpha$ can be patched together to obtain a global solution $\lambda$ for \eqref{DCOL_Mat}, hence implying the non-existence of a Nullstellensatz certificate of degree $d$. 

\subsection{Orderigs and Essential Graphs}

Each bijection from $[n]$ to the vertices of $G$ induces a monomial order $\preceq$ of $R_{V,k}$, namely the Graded Lexicographic Order (GLEX) where 
$$
x_n\preceq x_{n-1} \preceq \dots \preceq x_1.
$$
Given an edge $\{u,v\}\in E$ with $u<v$, we say that $v$ is a \textit{child} of $u$ and that $u$ is a \textit{parent} of $v$. Also, paths $P=u_1 u_2\dots u_t$ in $G$ satisfying $u_1<u_2<\dots<u_n$ will be called \textit{index-increasing paths}.  If there exists an index-increasing path from $u$ to $v$, we say that $v$ is a descendant of $u$. 
\begin{definition}
Let $x^{\aalpha}$ be a monomial with $\aalpha\in \Z_k^V$. The \textbf{descendant graph} of $x^{\aalpha}$ is the subgraph $H_{\aalpha}^{(0)}=(U_{\aalpha}^{(0)},F_{\aalpha}^{(0)})$ of $G$ induced by the vertices in the support of $\aalpha$ and their descendants.  
\end{definition}
The following lemmas show that due to the sparsity and $(k+1)$-colorability of $G$, the descendant graphs have a very special structure. 
\begin{lemma}\label{Lem:Ordering}
There exists a labelling of the vertex set $V$ such that every index-increasing path of $G$ has length at most $k$.
\end{lemma}
\begin{proof}
Consider any $(k+1)$-coloring of the graph $G$, say with colors $1, 2, \dots, k+1$. Label the vertices of $G$ in a way such that the inequality $u<v$ holds for every vertex $u$ in color class $i$ and every $v$ in color class $j$ where $i<j$. Formally, this can be done as follows. First, let $n_i\geq 0$ be the number of vertices in the color class $i\in [k+1]$ and set $n_0=0$. Then, assign to each vertex in the color class $i\in[k+1]$ a unique label in the set
$$
\left\{\sum_{j=0}^{i-1}n_j +1, \dots, \sum_{j=0}^{i}n_j\right\}.
$$
Since there are only $k+1$ colors, no index-increasing path will have length larger than $k$.  
\end{proof}

\begin{lemma}\label{Lem:H0Forest-k-version}
Let $G=(V,E)$ be a graph with chromatic number $k+1$ and girth $g> 2k$. Order the vertices of $G$ according to Lemma \ref{Lem:Ordering}. Let $x^{\aalpha}$ be a monomial of degree at most $|\aalpha|< \frac{g}{2k}-1$. Then, the descendant graph of $x^{\aalpha}$ is a forest.
\end{lemma}
\begin{proof}
%Let $v\in V$ be a vertex in the support of $\aalpha$. The distance from $v$ to its largest descendant is at most three. 
Suppose for the sake of a contradiction that $H_{\aalpha}^{(0)}$ has a cycle $C\subseteq  F_{\aalpha}^{(0)}$. Let us partition the cycle $C$ into index-increasing  paths $P_1, \dots, P_{s}$ and let $v_1, \dots, v_s$ be the vertices with the smallest index on each of these paths. Notice that the set $U:=\{v_1,\dots, v_s\}$ has at least $\frac{s}{2}$ vertices. Since the length of each  path $P_{i}$ is at most $k$, we have that $g\leq |C|\leq ks$, thus $|U|\geq \frac{|C|}{2k}$. 

Let $u\in V$ be a vertex in the support of $\aalpha$ and let $U(u)$ be the set of descendants of $u$ that lie in $U$. Since $G$ has girth $g>2k$ and has no index-increasing path of length larger than $k$,  any pair of vertices $v_i, v_j\in U(u)$ lie at distance at least $g-2k$ in $C$. In particular, $|U(u)|\leq \frac{|C|}{g-2k}$.
 
Now, take any set of $t$ vertices $u_1, \dots, u_t$ in the support of $\aalpha$ such that 
$$
U=\bigcup_{i=1}^t U(u_t).
$$
Then, 
$$
\frac{|C|}{2k}\leq |U|\leq \sum_{i=1}^{t}|U(u_i)|\leq t \frac{|C|}{g-2k} \quad \Longrightarrow \quad \frac{g}{2k}-1\leq t\leq |\aalpha|.
$$
\end{proof}
From now on, we will assume that the vertices of $G$ are ordered as in Lemma \ref{Lem:Ordering}, so that the descendant graphs of monomials of low degree are always sub-forests of $G$. The following lemma is the core of our argument.
\begin{lemma}\label{Lem:Core1}
Let $x^{\aalpha}$ be a monomial and let $H=(U,F)$ be its descendant graph. Suppose that no pair of vertices in different connected components of $H$ have common or adjacent parents in $G$. Then, for every sub-forest $H'=(U',F')$ of $G$ containing $H$, the monomial $x^{\aalpha}$ is reducible modulo $\cI_{F'}$ if and only if it is reducible modulo $\cI_{F}$.
\end{lemma}
\begin{proof}
Let $x^{\aalpha}$, $H$ and $H'$ be as in the statement. Clearly, if $x^{\aalpha}$ is reducible modulo $\cI_{F}$ then it is reducible modulo $\cI_{F'}$ as $F\subseteq F'$. Thus, let us assume that $x^{\aalpha}$ is reducible modulo $\cI_{F'}$.

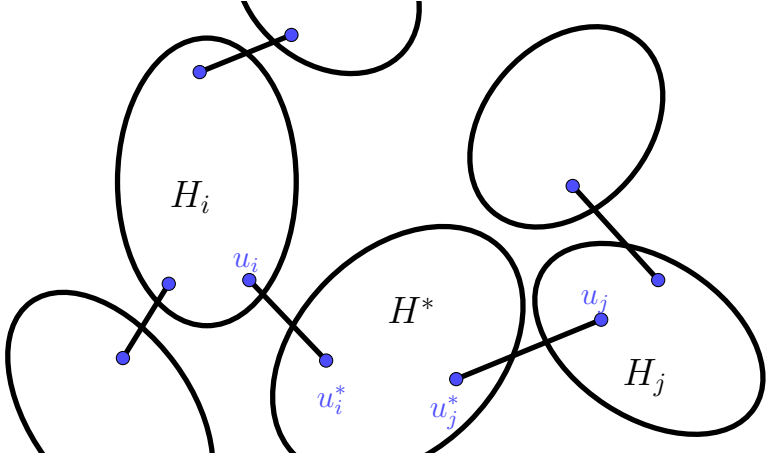
\begin{figure}
\begin{center}
\definecolor{ududff}{rgb}{0.30196078431372547,0.30196078431372547,1.}
\begin{tikzpicture}[line cap=round,line join=round,>=triangle 45,x=1.0cm,y=1.0cm]
\clip(4,-5) rectangle (14.5,1);
\draw [rotate around={90.:(6.8352817009455,-1.4011690205672984)},line width=2.pt] (6.8352817009455,-1.4011690205672984) ellipse (1.9134275536530514cm and 1.1879414981718077cm);
\draw [rotate around={45.:(9.374070463592895,-3.658873196217998)},line width=2.pt] (9.374070463592895,-3.658873196217998) ellipse (1.9417131566328973cm and 1.3305074154777548cm);
\draw [rotate around={-53.767581244552964:(5.55131677931586,-4.446077295155664)},line width=2.pt] (5.55131677931586,-4.446077295155664) ellipse (1.7900503802412233cm and 1.065917615860506cm);
\draw [rotate around={48.551733354820406:(11.61396012877168,-0.6689527278398852)},line width=2.pt] (11.61396012877168,-0.6689527278398852) ellipse (1.4960736484146675cm and 1.0876287792628467cm);
\draw [rotate around={-40.815083874881566:(8.432049048348064,1.2206056186763012)},line width=2.pt] (8.432049048348064,1.2206056186763012) ellipse (1.3522816868789cm and 1.0335210499394984cm);
\draw [line width=2.pt] (7.3994252805404575,-2.707283630559543)-- (8.420678734678361,-3.7779525744137983);
\draw [line width=2.pt] (5.719298630184551,-3.745008914602898)-- (6.328756336686204,-2.7566991202758935);
\draw [line width=2.pt] (7.959467497325762,0.5541386907195727)-- (6.74,0.06);
\draw [line width=2.pt] (12.83512914933898,-2.707283630559543)-- (11.698572885862927,-1.455424557745337);
\draw [rotate around={-33.190117042971686:(12.711590425048103,-3.4896955510684173)},line width=2.pt] (12.711590425048103,-3.4896955510684173) ellipse (1.6770783320545026cm and 1.048491362861788cm);
\draw [line width=2.pt] (10.15022087475062,-4.02503002299555)-- (12.077424973688277,-3.234382187533946);
\begin{scriptsize}
\draw[color=black] (6.61701336003158,-1.6036710268943877) node {{\Large $H_i$}};
\draw[color=black] (9.557234998154415,-3.1108434632430697) node {{\Large $H^\ast$}};
\draw [fill=ududff] (7.3994252805404575,-2.707283630559543) circle (2.5pt);
\draw[color=ududff] (7.358245705776833,-2.476678011883242) node {{\large $u_i$}};
\draw [fill=ududff] (8.420678734678361,-3.7779525744137983) circle (2.5pt);
\draw[color=ududff] (8.5,-4.3) node {{\large $u^\ast_i$}};
\draw [fill=ududff] (5.719298630184551,-3.745008914602898) circle (2.5pt);
\draw [fill=ududff] (6.328756336686204,-2.7566991202758935) circle (2.5pt);
\draw [fill=ududff] (7.959467497325762,0.5541386907195727) circle (2.5pt);
\draw [fill=ududff] (6.74,0.06) circle (2.5pt);
\draw [fill=ududff] (12.83512914933898,-2.707283630559543) circle (2.5pt);
\draw [fill=ududff] (11.698572885862927,-1.455424557745337) circle (2.5pt);
\draw[color=black] (12.662174935331755,-4) node {{\Large $H_{j}$}};
\draw [fill=ududff] (10.15022087475062,-4.02503002299555) circle (2.5pt);
\draw[color=ududff] (10,-4.464987806210247) node {{\large $u^\ast_j$}};
\draw [fill=ududff] (12.077424973688277,-3.234382187533946) circle (2.5pt);
\draw[color=ududff] (12,-3.) node {{\large $u_j$}};
\end{scriptsize}
\end{tikzpicture}
\end{center}
\caption{The graph of Lemma \ref{Lem:Core1}}\label{Fig:CoreLem1}
\end{figure}

Let $H^\ast=(U^\ast, F^\ast)$ be a connected component of the graph $H'\setminus H$. Since $H$ is a forest, there is at most one edge from $H^\ast$ to each component of $H$. Suppose that $H^\ast$ is connected to $\ell\geq 1$ components of $H$ and let $u_1^\ast, \dots, u_{\ell}^\ast\in U^\ast$ and $u_1, \dots, u_\ell\in U$ be such that $u_i u^\ast_i\in F'$ for each $i\in [\ell]$. Then, by our hypothesis on $H$,  for all $i\in [\ell]$ we have $u^\ast_i<u_i$ and no pair of the vertices $u_{i}^\ast, u_{j}^\ast$ or $u_i^{\ast}, u_{j}$ are adjacent for $i\neq j$.

Let $H''=(U'',F'')$ be the graph obtained from $H'$ after the deletion of the graph $H^\ast$.  Our goal is to show that $x^{\aalpha}$ is reducible modulo the ideal $\cI_{F''}$. Thus, by successively repeating this procedure with each of the remaining components of $H'\setminus H$, the reducibility of $x^{\aalpha}$ modulo $\cI_{F}$ follows.

Since $x^{\aalpha}$ is reducible modulo $\cI_{F'}$, it is the leading term of a polynomial $f$ of the form
\begin{equation}\label{Eq:pDef}
f(x)=\sum_{ u'v'\in F''}r_{u'v'}(x)q_{u'v'}(x)+\underbrace{\sum_{i=1}^{\ell}r_{u^\ast_i u_i}(x)q_{u^\ast_i u_i}(x)+\sum_{u'v'\in F^\ast}r_{u'v'}(x)q_{u'v'}(x)}_{:=p(x)},
\end{equation}
for some polynomials $r_{u'v'}$ with $u'v'\in F'$ of degree at most $d$. We will show that it is possible to transform $f$ into a polynomial $\tilde{f}\in \cI_{F''}$ whose leading monomial is $x^{\aalpha}$. We do this by analyzing the degrees of the $x_{u^{\ast}_i}$ variables appearing in $f$ for each $i\in [\ell]$. The reason behind such analysis is motivated by the following claim. 
\begin{claim}\label{Cl:k-3claim}
Let $f\in \cI_{F'}$ be as above. Suppose that for each $i\in [\ell]$ the $x_{u^\ast_i}$-degree of $f$ is at most $k-3$. Then, $x^{\aalpha}$ is reducible modulo $\cI_{F''}$.
\end{claim}  
\begin{proof}
Since $H^*$ is a tree, we can pick a partial coloring $b:=(b_w)_{w\in U^\ast\setminus\{u_1^\ast, \dots, u_\ell^\ast\}}$ of $H^\ast$ that colors all the neighbors in $H^\ast$ of each $u^\ast_i$ with the same color. Indeed, fix a primitive $k$-th root of the unity $\zeta$ and a vertex $v_0\in U^\ast$ of $H^\ast$. For every $w\in U^\ast$ let $d(w)$ be the distance in $H^\ast$ from $w$ to $v_0$. Then, define
$$
b_{w}:=\zeta^{d(w) \mod 2} \quad \quad  \forall w\in U^\ast\setminus\{u_1^\ast, \dots, u_\ell^\ast\}.
$$
Now, if a vertex $u^\ast\in U^\ast$ lies at distance $d(u^\ast)$ to $v_0$, then all of its neighbors satisfy $d(w)\equiv d(u^\ast)+1 \mod 2$ and $b_w$ will be the same for all of them. 

Consider the polynomial $f|_b$ obtained from $f$ after the evaluation of the partial coloring $b$. Notice that the leading term of $f|_b$ is still $x^{\aalpha}$ as no vertex of $H^\ast$ appears in the support of $\aalpha$. We claim that $f|_b\in \cI_{F''}$. Indeed, consider any coloring $a:=(a_{u''})_{u''\in F''}$ of $H''$ where each $a_{u''}$ is a $k$-th root of the unity and let $f|_{a,b}$ be the polynomial obtained from $f|_b$ after the evaluation of $a$. Then, $f|_{a,b}$ is a polynomial containing only $x_{u_{i}^\ast}$ variables and it vanishes on any coloring of $H'$ that agrees with $a$ and $b$. Now, the partial coloring induced by $a$ and $b$ colors the neighbors of each vertex $u_{i}^\ast$ with at most two colors. Thus, at least $k-2$ colors are available for each vertex $u_{i}^\ast$ to extend the partial coloring to a full coloring of $H'$ and obtain a root of $f|_{a,b}$. Since the $x_{u^\ast_i}$-degree of $f|_{a,b}$ is at most $k-3$ for each $i\in[\ell]$, this implies that $f|_{a,b}=0$ and the result follows. 
\renewcommand\qedsymbol{$\blacklozenge$}
\end{proof}

From the above claim, it is enough to reduce the $x_{u_i^\ast}$-degree of the polynomial $f$ for each $i\in [\ell]$. We start with the following simplification.
\begin{claim}\label{Cl:Polyf1}
We may assume that the polynomial $p$, defined in equation \eqref{Eq:pDef}, and the polynomials $r_{u'v'}$ with $u'v'\in F''$ have $x_{u^\ast_i}$-degree at most $k-2$ for every $i\in [\ell]$. In particular, this property holds for $f$ as well.
\end{claim}
\begin{proof}
Let us fix any index $i\in [\ell]$. If a term of the form  $c\cdot x_{u^\ast_i}^{k-1}x^{\bbeta}$ appears in some $r_{u'v'}$ with $u',v\in F''$, then we replace such a term by the polynomial $c\cdot[x_{u^\ast_i}^{k-1}-q_{u^\ast_i u_i}(x)]\cdot x^{\bbeta}$ in $r_{u'v'}$ and add the polynomial $c\cdot q_{u'v'}(x)\cdot x^{\bbeta}$ to $r_{u_i^\ast u_i}$. This way, we obtain a representation of $f$ such that all the $r_{u'v'}$ have $x_{u^\ast_i}$ degree at most $k-2$. 

Next, we replace any appearance of $x_{u_i^\ast}^{k-1}$ in the terms of $p(x)$ with the polynomial $x_{u^\ast_i}^{k-1}-q_{u^\ast_i u_i}(x)$. The resulting polynomial is still in the ideal generated by the polynomials $q_{u^\ast_i u_i}(x)$ with $i\in [\ell]$ and $q_{u'v'}$ with $u'v'\in F^\ast$. Moreover, since the leading term of $q_{u^\ast_i u_i}(x)$ is $x_{u^\ast_i}^{k-1}$, then the new monomials appearing are smaller than $x^{\aalpha}$ in the GLEX order. Indeed, no monomial of $p$ of the form $x_{u^\ast_i}^{k-1}\cdot x^{\bbeta}$ can cancel out with a term of $\sum_{ u'v'\in F''}r_{u'v'}(x)q_{u'v'}(x)$ as we have reduced the $x_{u_i^\ast}$-degree of each $r_{u'v'}$ with $\{u',v'\}\in F''$. Thus, such monomials would appear in $f$ as well, implying that $x_{u^\ast_i}^{k-1}\cdot x^{\bbeta}\preceq x^{\aalpha}$ and as a consequence, every term in $[x_{u^\ast_i}^{k-1}-q_{u^\ast_i u_i}(x)]\cdot x^{\bbeta}$ is smaller than $x^{\aalpha}$ in the GLEX order as well.
\renewcommand\qedsymbol{$\blacklozenge$}
\end{proof}

Our next goal is then to further reduce the degree of the $x_{u_i^\ast}$-variables. As in the proof of Claim \ref{Cl:k-3claim}, we can get rid of many of the terms involving some of the vertices of $H^\ast$ by using a partial coloring $b=(b_w)_{w\in U^\ast\setminus\{u_1^\ast, \dots, u_\ell^\ast\}}$  that colors all the neighbors in $H^\ast$ of each $u^\ast_i$ with the same color. Let us denote the color used by the neighbors of $u_i^\ast$ by $\zeta_{i}\in \K$ for each $i\in [\ell]$. Then, by evaluating the partial coloring $b$ on the polynomial $f$, we obtain a new polynomial $\tilde{f}$ whose leading monomial is still $x^{\aalpha}$. We can write $\tilde{f}$ as  
\begin{equation*}
\begin{aligned}
\tilde{f}(x)&=\sum_{ u'v'\in F''}\tilde{r}_{u'v'}(x)q_{u'v'}(x)+\sum_{i=1}^{\ell}\tilde{r}_{u^\ast_i u_i}(x)q_{u^\ast_i u_i}(x)+\sum_{i=1}^{\ell}\tilde{r}_i(x)q_{u^\ast_i u_i}(x_{u^\ast_i}, \zeta_i),\\
&=\sum_{ u'v'\in F''}\tilde{r}_{u'v'}(x)q_{u'v'}(x)+\underbrace{\sum_{i=1}^{\ell}\left[\tilde{r}_{u^\ast_i u_i}(x)q_{u^\ast_i u_i}(x)+\tilde{r}_i(x)q_{u^\ast_i u_i}(x_{u^\ast_i}, \zeta_i)\right]}_{:=\tilde{p}(x)},
\end{aligned}
\end{equation*}
for some polynomials $\tilde{r}_{u'v'}$ and $\tilde{r}_i$ of degree at most $d$. Notice that we have used the fact that all the neighbors in $H^\ast$ of each $u_i^\ast$ have been assigned the color $\zeta_i$, so that if $w\in U^\ast$ is a neighbor of $u^\ast_i$, then 
$$
q_{u^\ast_i w}(x)|_b=\sum_{r=0}^{k-1}x_{u^\ast_i}^r \zeta_i^{k-1-r}=q_{u^\ast_i u_i}(x_{u^\ast_i}, \zeta_i).
$$
Now, for each $i\in [\ell]$ let us define the polynomial $t_i(x_{u^\ast_i}, x_{u_i})$ given by the equation
\begin{equation}
q_{u^\ast_i u_i}(x_{u^\ast_i}, x_{u_i})-q_{u^\ast_i u_i}(x_{u^\ast_i}, \zeta_i)=:(x_{u_i}-\zeta_i)\cdot t_i(x_{u^\ast_i}, x_{u_i}).
\end{equation}
Notice that the leading monomial of each $t_i(x)$ is $x_{u^\ast_i}^{k-2}$. Moreover, for any $k$-th root of the unity $\zeta\neq \zeta_i$ we have 
\begin{equation}\label{Eq:Polyt}
\langle q_{u^\ast_i u_i}(x_{u^\ast_i}, \zeta_i),q_{u^\ast_i u_i}(x_{u^\ast_i}, \zeta)\rangle=\langle t(x_{u^\ast_i}, \zeta)\rangle. 
\end{equation}
We will successively reduce the $x_{u^\ast_i}$-degree of $\tilde{f}$ for each $i\in [\ell]$ as follows. First, set $f^{(0)}:=\tilde{f}$, $p^{(0)}:=\tilde{p}$ and $r_{u'v'}^{(0)}:= r_{u'v'}$ for $\{u'v'\}\in F''$. Then, for each $i\in [\ell]$ and $\{u'v'\}\in F''$ write
\begin{align*}
p^{(i-1)}(x)&=x_{u^\ast_{i}}^{k-2}s^{(i-1)}(x)+\text{other terms with $x_{u^\ast_i}$-degree $< k-2$,}\\
r^{(i-1)}_{u'v'}(x)&=x_{u^\ast_{i}}^{k-2}r^{(i-1,0)}_{u'v'}(x)+\text{other terms with $x_{u^\ast_i}$-degree $< k-2$,}
\end{align*}
and define the polynomials 
\begin{align*}
p^{(i)}(x)&:=p^{(i-1)}(x)-t_i(x)s^{(i-1)}(x),\\
r^{(i)}_{u',v'}(x)&:=r^{(i-1)}_{u'v'}(x)-t_i(x)r_{u'v'}^{(i-1,0)}(x), \\
f^{(i)}(x)&:=\sum_{u'v'\in F''}r^{(i)}_{u',v'}(x)q_{u'v'}(x)+p^{(i)}(x).
\end{align*}
Notice that the degree of each $r_{u'v'}^{(i)}$ is at most $d$. Moreover, we have the following:
\begin{claim}
For every $i\in [\ell]$, the leading monomial of $f^{(i)}$ is $x^{\aalpha}$.
\end{claim}
\begin{proof}
We prove this by induction on $i$ with the case $i=0$ being trivial. Now, suppose that the leading term of $f^{(i-1)}$ is $x^{\aalpha}$. Since the polynomials $q_{u'v'}$ are free of $x_{u^\ast_i}$-variables for $\{u'v'\}\in F''$,  we can write 
\begin{align*}
f^{(i-1)}(x)&=\sum_{u'v'\in F''}r^{(i-1)}_{u'v'}(x)q_{u'v'}(x)+p^{(i-1)}(x), \\
&=\sum_{u'v'\in F''}\left(x_{u^\ast_{i}}^{k-2}r^{(i-1,0)}_{u'v'}(x)+\cdots\right)q_{u'v'}(x)+\left(x_{u^\ast_{i}}^{k-2}s^{(i-1)}(x)+\cdots\right), \\
&=x_{u^\ast_i}^{k-2}\left(\sum_{u'v'\in F''}r_{u'v'}^{(i-1,0)}(x)q_{u'v'}(x)+s^{(i-1)}(x)\right)+\cdots, 
\end{align*}
where the three dots consist of terms with $x_{u^\ast_i}$-degree less than $k-2$, all of them smaller than $x^{\aalpha}$ in GLEX order. However, by the definition of $f^{(i)}$ we have 
\begin{align*}
f^{(i)}(x)&=f^{(i-1)}(x)-t_i(x)\left(\sum_{u'v'\in F''}r_{u'v'}^{(i-1,0)}(x)q_{u'v'}(x)+s^{(i-1)}(x)\right), \\
&=(x_{u^\ast_i}^{k-2}-t_i(x))\left(\sum_{u'v'\in F''}r_{u'v'}^{(i-1,0)}(x)q_{u'v'}(x)+s^{(i-1)}(x)\right)+\cdots.
\end{align*}
In other words, $f^{(i)}$ is obtained by replacing any appearance of the monomial $x_{u_i^\ast}^{k-2}$ in $f^{(i)}$ with the polynomial $(x_{u_i^\ast}^{k-2}-t_i(x))$. This operation does not affect $x^{\aalpha}$ as no vertex in $H^\ast$ is in the support of $\aalpha$. Moreover, since the leading term of $t_i$ is precisely  $x_{u_i^\ast}^{k-2}$, the new monomials appearing in $f^{(i)}$ are smaller than $x^{\aalpha}$ in GLEX order. 
\renewcommand\qedsymbol{$\blacklozenge$} 
\end{proof}
\begin{claim}
$p^{(\ell)}(x)=0$.
\end{claim}
\begin{proof}
Let $a:=(a_{w})_{w\in U''}$ be any sequence with $a_{w}^k=1$ for all $w\in U''$ and for every $i\in [\ell]$ let $p^{(i)}|_a$ be the polynomial obtained after the evaluation $x_w=a_w$ for all $w \in U''$. Let us first show that for every $i\in \{0,1,\dots, \ell-1\}$ we have 
\begin{equation}\label{Eq:p=0mod}
p^{(i)}|_{a}\in \langle q_{u_j^\ast u_j}(x_{u_j^\ast}, a_{u_j}), q_{u_j^\ast u_j}(x_{u^\ast_j}, \zeta_j): \ j>i\rangle.
\end{equation}
Indeed, for $i=0$ the statement holds from the definition of the polynomial $p^{(0)}$. Thus, suppose that the statement holds for $p^{(i-1)}$. In particular, from the definition of $p^{(i)}$ we have that 
$$
p^{(i)}|_{a}\in \langle t_i(x_{u^\ast_i}, a_{u_i}), q_{u_j^\ast u_j}(x_{u_j^\ast}, a_{u_j}), q_{u_j^\ast u_j}(x_{u^\ast_j}, \zeta_j): \ j\geq i\rangle
$$ 
and the $x_{u_i^\ast}$-degree of $p^{(i)}|_a$ is at most $k-3$. Let $c=(c_{u^\ast_j})_{j>i}$ be any vanishing point of the ideal described in equation \eqref{Eq:p=0mod}, in other words each $c_{u^\ast_j}$ is any root of the unity different from $a_{u_i}$ and $\zeta_i$. Let $p^{(i)}|_{a,c}$ be the polynomial obtained after the evaluation by $c$, so that 
$$
p^{(i)}|_{a,c}\in \langle t_i(x_{u^\ast_i}, a_{u_i}), q_{u_i^\ast u_j}(x_{u_i^\ast}, a_{u_i}), q_{u_i^\ast u_i}(x_{u^\ast_i}, \zeta_i)\rangle.
$$
We see that $p^{(i)}|_{a,c}$ is a polynomial of degree at most $k-3$ with at least $k-2$ roots, namely any $k$-th root of the unity $\zeta$ different from $\zeta_i$ and $a_{u_i}$ makes the polynomials $t_i(x_{u^\ast_i}, a_{u_i}), q_{u_i^\ast u_j}(x_{u_i^\ast}, a_{u_i})$ and  $q_{u_i^\ast u_i}(x_{u^\ast_i}, \zeta_i)$ vanish. Since the point $c$ was arbitrary, the equation \eqref{Eq:p=0mod} is proven for $i$. 

From the above and by the definition of $p^{(\ell)}$ we conclude that 
$$
p^{(\ell)}|_a(x)\in \langle t_\ell(x_{u^\ast_\ell}, a_{u_\ell}), q_{u_\ell^\ast u_\ell}(x_{u_\ell^\ast}, a_{u_\ell}), q_{u_\ell^\ast u_\ell}(x_{u^\ast_\ell}, \zeta_\ell)\rangle.
$$
Since the $x_{u^\ast_\ell}$-degree of $p^{(\ell)}$ is at most $k-3$, via a similar argument, we conclude that $p^{(\ell)}|_a$ vanishes for every possible $a$ and the result follows. 
\renewcommand\qedsymbol{$\blacklozenge$} 
\end{proof}
The claims above show that the polynomial 
$$
f^{(\ell)}(x)=\sum_{u'v'\in F''}r^{(\ell)}_{u',v'}(x)q_{u'v'}(x)\in \cI_{F''}
$$
has leading monomial $x^{\aalpha}$ and the result follows.

\renewcommand\qedsymbol{$\square$}
\end{proof}

\begin{remark}\label{Rem:Core1}
We have shown an even stronger result. Under the hypothesis of Lemma \ref{Lem:Core1}, if $x^{\aalpha}$ is the leading term of a polynomial of the form $\sum_{uv\in F'} r_{u'v'} q_{u'v'}$ where the polynomials $r_{u'v'}$ all have degree at most $d$, then $x^{\aalpha}$ is the leading term of polynomial of the form $\sum_{uv\in F_{\aalpha}^{(0)}} \tilde{r}_{u'v'} q_{u'v'}$ where each $\tilde{r}_{u'v'}$ has degree at most $d$ as well.
\end{remark}

As the following example shows, Lemma \ref{Lem:Core1} might not be true when the connected components of the descendant graph of $x^{\aalpha}$ have common or adjacent parents. So, this assumption cannot be relaxed without changing the rest of the statement.
\begin{example}\label{Ex1}
Let $k=3$ and consider the tree $G=(V,E)$ depicted in Figure \ref{Fig:Ex1}. Consider the monomial $x^{\aalpha}:=x_5^2x_6x_7x_8^2x_9$, we claim that $x^{\aalpha}$ is reducible modulo $\cI_{E}$, but it is irreducible modulo $\cI_{F}$ for any proper subset of edges $F\subseteq E$. Indeed, consider the polynomials
\begin{align*}
f_1(x)&=(x_8-x_{9})(x_8-x_{10})(x_9-x_{10}),\\
f_2(x)&=(x_5-x_7)(x_6-x_7)(x_{7}-x_{8}),\\
f_3(x)&=(x_5-x_6)
\end{align*}
\begin{figure}
\begin{center}
\definecolor{ududff}{rgb}{0.30196078431372547,0.30196078431372547,1.}
\begin{tikzpicture}[line cap=round,line join=round,>=triangle 45,x=1.0cm,y=1.0cm]
\clip(0.5,-1.6) rectangle (11.52710711455958,3.600262576981546);
\draw [line width=2.pt] (1.,3.)-- (2.,1.);
\draw [line width=2.pt] (2.,1.)-- (3.,3.);
\draw [line width=2.pt] (2.,1.)-- (6.,-1.);
\draw [line width=2.pt] (5.,3.)-- (6.,1.);
\draw [line width=2.pt] (7.,3.)-- (6.,1.);
\draw [line width=2.pt] (9.,3.)-- (10.,1.);
\draw [line width=2.pt] (11.,3.)-- (10.,1.);
\draw [line width=2.pt] (6.,1.)-- (6.,-1.);
\draw [line width=2.pt] (10.,1.)-- (6.,-1.);
\begin{scriptsize}
\draw [fill=ududff] (1.,3.) circle (2.5pt);
\draw[color=ududff] (1.1429146508838257,3.35181619611292) node {5};
\draw [fill=ududff] (3.,3.) circle (2.5pt);
\draw[color=ududff] (3.1400787576315667,3.35181619611292) node {6};
\draw [fill=ududff] (5.,3.) circle (2.5pt);
\draw[color=ududff] (5.137242864379308,3.35181619611292) node {7};
\draw [fill=ududff] (7.,3.) circle (2.5pt);
\draw[color=ududff] (7.134406971127049,3.35181619611292) node {8};
\draw [fill=ududff] (9.,3.) circle (2.5pt);
\draw[color=ududff] (9.13157107787479,3.35181619611292) node {9};
\draw [fill=ududff] (11.,3.) circle (2.5pt);
\draw[color=ududff] (11.186345687701793,3.35181619611292) node {10};
\draw [fill=ududff] (2.,1.) circle (2.5pt);
\draw[color=ududff] (2.4487527206804254,1.1242100770481311) node {2};
\draw [fill=ududff] (6.,1.) circle (2.5pt);
\draw[color=ududff] (6.423877433149488,0.8553610626782429) node {3};
\draw [fill=ududff] (10.,1.) circle (2.5pt);
\draw[color=ududff] (10.187763634327922,0.8745645637046633) node {4};
\draw [fill=ududff] (6.,-1.) circle (2.5pt);
\draw[color=ududff] (5.982196909541814,-1.4682625615186493) node {1};
\end{scriptsize}
\end{tikzpicture}
\end{center}
\caption{The graph of Example \ref{Ex1}}\label{Fig:Ex1}
\end{figure}
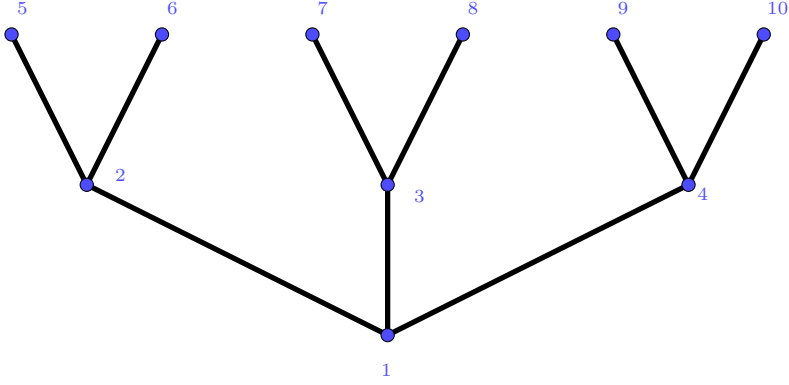
Notice that the leading term of the product $f:=f_1f_2f_3$ is $x^{\aalpha}$. We claim that $f$ vanishes on all possible colorings of $G$. Indeed, let $a:=(a_v)_{v\in V}$ be any coloring of the graph $G$ where each $a_v$ is a $3$-rd root of the unity. If $f_1(a)\neq 0$, then  $a_8,a_9$ and $a_{10}$ are pairwise distinct. Since $4$ is adjacent to both $9$ and $10$, this also implies that $a_4=a_8$ and that $a_3\neq a_4$. If in addition $f_2(a)\neq 0$ then $a_7\neq a_8$ and both $a_5$ and $a_6$ are different from $a_7$. Since $3$ is adjacent to both $7$ and $8$ and $a_4=a_8$, this implies that $a_5,a_6\in \{a_3,a_4\}$. Thus, we conclude that $a_5=a_6$ and $f_3(a)=0$. Otherwise, $a_2, a_3$ and $a_4$ would be pairwise distinct, which cannot happen as all of them have a neighbor in common.  This shows the reducibility of $x^{\aalpha}$ modulo $\cI_E$. 

By the symmetry of the graph $G$ and the way we have enumerated its vertices, it is not hard to see that the irreducibility of $x^{\aalpha}$ modulo $\cI_F$ for any $F\subsetneq E$ follows from the following claim. 
\begin{claim}\label{Cl:Exa}
$x^{\aalpha}$ is irreducible modulo $\cI_{F}$ for every set of the form $F=E\setminus \{u,v\}$ with $$\{u,v\}\in\{\{1,4\},\{2,6\},\{3,8\},\{4,10\}\}.$$
\end{claim}
The claim above can be verified with aid of a computer algebra system such as Macaulay2 \cite{M2} by calculating a Gr\"{o}bner basis of each of the four ideals $\cI_{F}$ above. The details have been included in Appendix \ref{AppendixA}.  
\end{example}

The above example motivates the following definition:
\begin{definition}\label{Def:EssentialGraphs}
Let $x^{\aalpha}$ be a monomial in $R_{V,k}$. The \textbf{essential graph} of $x^{\aalpha}$ is the subgraph $H_{\aalpha}:=(U_{\aalpha}, F_{\aalpha})$ of $G$ constructed as follows:
\begin{enumerate}
\item Initially, set $H_{\aalpha}:=H_{\aalpha}^{(0)}$ to be the descendant graph of $x^{\aalpha}$.
\item Let $U^\ast \subseteq U_{\aalpha}$ be set of parents of the vertices in $U_{\aalpha}$. If a pair of vertices $u,v\in U_{\aalpha}$ with parents $u^\ast, v^\ast\in U^\ast$ satisfies either $u^\ast=v^\ast$ or $u^\ast v^\ast\in E$, then we add the vertices $u^\ast,v^\ast$ to $U_{\aalpha}$ along with all of their descendants. Then, we update $F_{\aalpha}$ to be the graph induced by this new set $U_{\aalpha}$.
\item We repeat step 2. until no pair of connected components of $H_{\aalpha}$ have  common or adjacent parents.
\end{enumerate} 
\end{definition}

\begin{example}
Consider the graph $G=(V,E)$ of Example \ref{Ex1} and the monomial $x^{\aalpha}:=x_5^2x_6x_7x_8^2x_9$. Then, the descendant graph $H_{\aalpha}^0$ consist of all the leaves of the tree $G$, whereas the essential graph $H_{\aalpha}$ is the entire graph $G$. Recall that $x^{\aalpha}$ is irreducible modulo $\cI_{F}$ for every subset of edges $F\subseteq E$, while being reducible modulo $\cI_{E}$. 
\end{example}

\begin{corollary}\label{Cor:H1Forest-k-version}
Let $G$ be as above and let $\aalpha$ be a multi-index of degree at most $d$. Suppose that $2d<\frac{g}{2k}-1$, then the essential graph of $x^{\aalpha}$ is a forest. 
\end{corollary}
\begin{proof}
Set initially $x^{\bbeta}:=x^{\aalpha}$.  At each step of the construction of $H_{\aalpha}$, if parents $u^\ast$ and $v^\ast$ with $u^\ast\leq v^\ast$ are added to the graph, then update $x^{\bbeta}:=x^{\bbeta}x_{u^\ast}$. Thus, at the end of the construction of $H_{\aalpha}$, the descendant graph of $x^{\bbeta}$ equals $H_{\aalpha}$.

Now, at each step in the construction of $H_{\aalpha}$ we are reducing its number of connected components. Thus,  the degree of $x^{\bbeta}$ at the end of the construction is at most $2d$. By Lemma \ref{Lem:H0Forest-k-version}, $H_{\aalpha}=H_{\bbeta}^{(0)}$ is a forest.
\end{proof}
\begin{corollary}\label{Cor:Core2}
Let $x^{\aalpha}$ be a monomial whose essential graph $H_{\aalpha}$ is a forest and let $H'=(U',F')$ be a larger forest containing $H_{\aalpha}$. Then, $x^{\aalpha}$ is reducible modulo $\cI_{F_{\aalpha}}$ if and only if it is reducible modulo $\cI_{F'}$.
\end{corollary}

%Main Theorem
\section{Main Theorem}

Before going into the proof of Theorem \ref{Thm:Main}, let us recall some basic notation. For every subset of edges $F\subseteq E$ let us denote by
\begin{equation}\label{Eq:MatrixF}
\hat{A}_{F,d} \lambda =\hat{c}_{F,d}
\end{equation}
the system of linear equations \eqref{DCOL} for the graph induced by $F$. Notice that the system \eqref{Eq:MatrixF} has a solution for every $d\geq 0$ whenever $F$ is $k$-colorable. In particular, this holds whenever $F$ is a forest. 

By looking at the system  \eqref{DCOL}, one sees that the columns of $\hat{A}_{F,d}$ can be indexed by monomials $x^{\aalpha}$ with $\aalpha\in \Z_k^V$ of degree at most $d+k-1$. We will assume that these columns are ordered using the GLEX order from left to right,  where the largest monomials are the left most columns in $\tilde{A}_{F,d}$.

As it is custom in linear programming, let us call a set of monomials $\cB$ a \textit{basis} of  $\hat{A}_{F,d}$ if the corresponding columns of $\hat{A}_{F,d}$ form a basis for its column space. If the system \eqref{Eq:MatrixF} has a solution,  every basis $\cB$ induces a corresponding \textit{basic solution}, namely by setting $\lambda_{\aalpha}=0$ for all $x^{\aalpha}\notin \cB$ and solving the resultant system of equations with a unique solution.

\begin{lemma}
For every set of edges $F\subseteq E$, let $\cB_{F,d}$ be the set of leading monomials of polynomials of the form $\sum_{uv\in F}r_{uv}q_{uv}$ where each $r_{uv}$ has degree at most $d$. Then, $\cB_{F,d}\cup\{1\}$ is a basis for the matrix $\hat{A}_{F,d}$.
\end{lemma} 
\begin{proof}
Using the indexing on the columns described above, we can identify each row of $\hat{A}_{F,d}$ with a polynomial in $R_{V,k}$. In fact these polynomials are either the constant polynomial $1$ or polynomials of the form
$$
x^{\aalpha}q_{uv}(x), \quad |\aalpha|\leq d, \ \{u,v\}\in F.
$$ 
Thus, the row space of $\hat{A}_{F,d}$ corresponds precisely with the space of polynomials of the form $\sum_{uv\in F}r_{uv}(x)q_{uv}(x)$ where each $r_{uv}$ has degree at most $d$. Let $R$ be the row-reduced echelon form of $\tilde{A}_{F,d}$  and let $\cB$ be the basis corresponding to the leading ones of $R$. We claim that $\cB=\cB_{F,d}
\cup \cB$. Indeed, since we have ordered the columns using the GLEX order, the leading terms of polynomials in the non-zero rows of $R$ correspond to principal ones of $R$ and $\cB\subseteq \cB_{F,d}\cup\{1\}$. Conversely, if $x^{\aalpha}$ is the leading monomial of the polynomial $f$ in the row-span of $\hat{A}_{F,d}$, then  we should be able to write $f$ as linear combination of polynomials represented by rows of $R$. However, in such linear combination no cancellation of leading ones can occur and the leading term of $f$ should be a monomial in $\cB$.  
\end{proof}

By Remark \ref{Rem:Core1}, we can rewrite corollary \ref{Cor:Core2} as follows.
\begin{corollary}\label{Cor:Core3}
Let $x^{\aalpha}$ be a monomial whose essential graph $H_{\aalpha}=(U_{\aalpha}, F_{\aalpha})$ is a forest and let $H'=(U',F')$ be a larger forest containing $H_{\aalpha}$. Then, for any $d\geq 0$ and any $x^{\bbeta}$ with $\supp(\bbeta)\subseteq U_{\aalpha}$ 
$$
x^{\bbeta} \in \cB_{F',d}\Rightarrow x^{\bbeta}\in \cB_{F_{\aalpha},d}.
$$
\end{corollary}
\begin{proof}
Since $\supp(\bbeta)\subseteq U_{\aalpha}$, the essential graph $H_{\bbeta}$ of $x^{\bbeta}$ is a subforest of $H_{\aalpha}$. This follows from the fact that $H_{\aalpha}$ is closed under descendants and common or adjacent ancestors. In particular, $H_{\bbeta}$ is subforest of $H'$ as well. By Corollary \ref{Cor:Core2} and Remark \ref{Rem:Core1}, if $x^{\bbeta}$ is the leading term of a polynomial of the form $\sum_{uv\in F''}r_{uv}q_{uv}$ where each $r_{uv}$ has degree at most $d$, then $x^{\bbeta}$ is the leading term of a polynomial of the form $\sum_{uv\in F_{\bbeta}}\tilde{r}_{uv}q_{uv}$ where each $\tilde{r}_{uv}$ has degree at most $d$ as well.
\end{proof}
We are ready to prove our main result.
\begin{proof}[Proof of Theorem \ref{Thm:Main}]
Let $d\geq 0$ be such that $2(d+k-1)<\frac{g}{2k}-1$. Then, for every monomial $x^{\aalpha}$ of degree at most $d+k-1$, its essential graph $H_{\aalpha}=(U_{\aalpha}, F_{\aalpha})$ is a forest. In particular, the system 
\begin{equation}\label{Eq:Main1}
\hat{A}_{F_{\aalpha}, d} \cdot\mu = \hat{c}_{F_{\aalpha},d}
\end{equation}
has a solution. Let $\mu^{(\aalpha)}$ be the basic solution of \eqref{Eq:Main1} corresponding to the basis $\cB_{F_{\aalpha}, d}$ and set $\lambda_{\aalpha}:=\mu^{(\aalpha)}_{\aalpha}$. Notice that the essential graph of the constant polynomial $1$ has no edges, thus the system \eqref{Eq:Main1} has only one equation, namely $\mu_{0}=1$ and as a consequence $\lambda_0=1$. We claim that $\lambda=(\lambda_{\aalpha})_{|\aalpha|\leq d+k-1}$ is a Dual Nullstellensatz Certificate of degree $d$, i.e., $\lambda$ is a solution to the system 
\begin{equation}\label{Eq:Main2}
\hat{A}_{E, d} \cdot\lambda = \hat{c}_{E,d}.
\end{equation}
Indeed, let $x^{\aalpha}$ be a monomial of degree $|\aalpha|\leq d$ and let $\{u,v\}\in E$ be any edge of $G$. Our goal is to show that
\begin{equation}\label{Eq:Main3}
\sum_{r\in \Z_k} \lambda_{\aalpha+r(e_u-e_v)-e_v}=0.
\end{equation}
Let $r\in \Z_k$ be such that the $u$-th and $v$-th coordinates of $\bbeta:=\aalpha+r(e_u-e_v)-e_v$ are non-zero. In other words, $r$ is such that the support of $\bbeta$ is maximal among all the multi-indices appearing in \eqref{Eq:Main3}. In particular, for any other $r'\in \Z_k$ and $\eeta:=\aalpha+r'(e_u-e_v)-e_v$ we have $\supp(\eeta)\subseteq \supp(\bbeta)\subseteq U_{\bbeta}$. 

Let $R_{\eeta}$ and $R_{\bbeta}$ be the row-reduced echelon forms of $\tilde{A}_{F_{\eeta},d}$ and $\tilde{A}_{F_{\bbeta},d}$ respectively. We claim that the rows of $R_{\eeta}$ are rows of  $R_{\bbeta}$ as well. Indeed, the rows of $\tilde{A}_{F_{\eeta},d}$ are rows of $\tilde{A}_{F_{\bbeta},d}$ and as a consequence every row in $R_{\eeta}$ is in the row span of the rows of $R_{\bbeta}$. However, by Corollary \ref{Cor:Core3}, every column of $R_{\eeta}$ corresponding to a principal one of $R_{\bbeta}$  is also a principal one of $R_{\eeta}$. Thus, each row of $R_{\eeta}$ cannot be obtained by non-zero combination of two or more different rows of $R_{\bbeta}$. 

Since every row of $R_{\eeta}$ appears in $R_{\bbeta}$, for every column $\eeta'$ of $R_{\eeta}$ we have $\mu^{(\eeta)}_{\eeta'}=\mu^{(\bbeta)}_{\eeta'}$. In particular, $\mu^{(\eeta)}_{\eeta}=\mu^{(\bbeta)}_{\eeta}$ and 
\begin{align*}
\sum_{r\in \Z_k} \lambda_{\aalpha+r(e_u-e_v)-e_v}&=\sum_{r\in \Z_k} \mu_{\aalpha+r(e_u-e_v)-e_v}^{(\aalpha+r(e_u-e_v)-e_v)}, \\
&=\sum_{r\in \Z_k}\mu_{\aalpha+r(e_u-e_v)-e_v}^{(\bbeta)}=0
\end{align*}
as desired.

\end{proof}
%\begin{lemma}\label{Lem:ModIVk}
%Let $G$ be a non-$k$-colorable graph. Then the system \eqref{BCOL_k} has a Nullstellensatz Certificate of degree $d$ if and only if 
%\begin{equation}
%\sum_{\{u,v\}\in E}r_{uv}(x) q_{uv} (x)\equiv 1 \mod \cI_{V,k} 
%\end{equation}
%for some polynomials $r_{uv}\in \K[x_u: u\in V]/\cI_{V,k}$ of degree at most $d$.
%\end{lemma}

%Final Remarks
\section{Concluding Remarks}

In this article we have studied the behavior of the Nullstellensatz and Polynomial Calculus approach to graph $k$-colorability for graphs having large girth. We showed that  as the girth of a non-$k$-colorable graph increases, the degrees of the Nullstellensatz certificates must grow as well. This was obtained by studying the structure of the principal ideals generated by polynomials in Bayer's formulation corresponding to sub-forests of the graph and applying a general technique introduced by Aleknovich and Razborov \cite{AR03}.  

In the words of Aleknovich and Razborov, informally, "\textit{everything we can infer in small degree we can also infer locally}". This is precisely what motivated our work: if a non-$k$-colorable graph $G$ has a small Nullstellensatz certificate, then one should be able to detect its non-$k$-colorability by looking at the local structure of $G$. We observed that if the essential graph of monomials of low degree were forests, then it was possible to build dual Nullstellensatz Certificates in a local fashion. One of our future goals is to understand whether this sparsity property of the essential graphs can be further extended, say to essential graphs that are not trees, but other class of graphs such as bipartite. 
%\Jnote{Think of a well posed example... triangle-free graphs, Mycielskian graphs?}.

One of the reasons why the Nullstellensatz method is appealing is that the linear systems used to find certificates of non-$k$-colorability using Bayer's formulation are quite sparse. In addition,  computations over finite fields are possible and in many cases can be carried out very efficiently. For instance, detecting non-$3$-colorability can be done by solving linear systems over $\F_2$. Thus,  in principle, it may be possible to use methods that exploit the sparseness of the system such as Coppersmith's Block Wiedemann or Block Lanczos Methods which work on finite field algebra. Although, implementations of the Nullstellensatz method exist \cite{Margulies08}, to the best of our knowledge, an implementation using the aforementioned techniques is not available to the public. 

The problem of characterizing when the Nullstellensatz method effectively certifies non-$k$-colorability is wide open. For the case $k=3$, De Loera et al. \cite{DeLoera10} obtained a characterization of all non-$3$-colorable graphs having degree one Nullstellensatz Certificate over $\F_2$. However, we do not know what classes of non-$3$-colorable graphs admit a degree four Nullstellensatz Certificate.  

\begin{openproblem}
Characterize all non-$3$-colorable graphs whose Bayer's formulation requires a Nullstellensatz certificates of degree at most four over $\F_2$.
\end{openproblem}  

Even simpler questions like determining the size of the smallest degree of a Nullstellensatz certificate for proving the non-$k$-colorability of the complete graph $K_{k+1}$ is open for general $k$. De Loera et al. \cite{loeraetal14} obtained computational results for $K_{k+1}$ with $k\leq 10$ over fields $\F_q$ with $q\in \{2,3,5,7\}$. We do not know the exact minimum degrees for $k\geq 8$.  

\begin{openproblem}
Let $p$ be a prime and let $k\geq 8$ be relatively prime to $p$. Find the smallest degree Nullstellensatz Certificate for proving the non-$k$-colorability of the complete graph $K_{k+1}$ over $\F_p$.
\end{openproblem}

Another interesting line of research is to study how the Nullstellensatz method behaves with respect to graph operations such as the Haj\'{o}s construction and other similar operations. Recall that any $(k+1)$-critical graph $G$, i.e., a graph $G$ such that $\chi(G)=k+1$, but $\chi(H)<k+1$ for every proper subgraph $H\subseteq G$, can be obtained from $K_{k+1}$ using repeated iterations of the Haj\'{o}s construction. Since the Nullstellensatz Certificates for detecting the non-$k$-colorability of $(k+1)$-critical graphs are not universally bounded, the following question arises. 
\begin{openproblem}\cite{Omar15}
Let $G_1$ and $G_2$ be $(k+1)$-critical graphs and let $G$ constructed from $G_1$ and $G_2$ using the Haj\'os Construction. What is the relationship between the minimum degree Nullstellensatz certificates of $G_1$, $G_2$ and $G$?  
\end{openproblem}  

Finally, it is our general belief that, if a non-$k$-colorable graph $G$ has a small degree Nullstellensatz certificate, then one should be able to detect its non-$k$-colorability by looking at the local structure of the graph. Our results follow this line of reasoning by exploiting the fact that the \textit{essential graphs} of monomials of low degree were forests for graphs of high girth to build dual certificates. 
\begin{openproblem}
What families of graphs admit an ordering of its vertices in such a way that the \textit{essential graphs} of monomials of low degree are forests? 
\end{openproblem}
In addition, it may be interesting to see if our methods can be extended to the case in which the essential graphs of monomials of low degree are not forests, but  another family of graphs whose chromatic number is easy to calculate, such as bipartite graphs.   

%Can we characterize all non-$3$-colorable graphs whose Bayer's formulation requires Nullstellensatz certificates of degree at most four? Is the Nullstellensatz method efficient for detecting the chromatic number of perfect graphs? We do not know of a family of planar non-$3$-colorable graphs that require large Nullstellensatz certificates. Even simple questions like determining the size of the smallest Nullstellensatz degree certificate for proving the non-$k$-colorability of $K_{k+1}$ is open for general $k$  

\appendix
\section{Proof of Claim \ref{Cl:Exa} }
\label{AppendixA}
In this appendix we provide computational certificates for Claim \ref{Cl:Exa} using \verb`Macaulay2`. For this, we simply have check that the monomial $x^{\aalpha}=x_5^2x_6x_7x_8^2x_9$ is not divisible by any leading term in a Gr\"{o}bner basis of the ideals $\cI_{F_1}, \cI_{F_2}, \cI_{F_3}$ and $\cI_{F_4}$, where $F_1=E\setminus\{1,4\}$, $F_2=E\setminus \{2,6\}$, $F_3=E\setminus\{3,8\}$ and $F_4=E\setminus\{4,10\}$. The output provided by Macaulay2 is the following:
\begin{enumerate}
    \item[($\cI_{F_1}$)] \begin{verbatim}
Leading Monomials for I_{F_1}
      {-2} | x_4x_9      |
      {-2} | x_4^2       |
      {-2} | x_3x_7      |
      {-2} | x_3^2       |
      {-2} | x_2x_5      |
      {-2} | x_2^2       |
      {-2} | x_1x_2      |
      {-2} | x_1^2       |
      {-3} | x_9^3       |
      {-3} | x_7^3       |
      {-3} | x_6^3       |
      {-3} | x_5^3       |
      {-3} | x_1x_3x_6   |
      {-3} | x_1x_3x_5   |
      {-4} | x_1x_6^2x_7 |
      {-4} | x_1x_5^2x_7 |
      {-4} | x_1x_5^2x_6 |
\end{verbatim}

\item[($\cI_{F_2}$)]
\begin{verbatim}
Leading Monomials for I_{F_2}
      {-2} | x_4x_9            |
      {-2} | x_4^2             |
      {-2} | x_3x_7            |
      {-2} | x_3^2             |
      {-2} | x_2^2             |
      {-2} | x_1x_3            |
      {-2} | x_1x_2            |
      {-2} | x_1^2             |
      {-3} | x_9^3             |
      {-3} | x_8^3             |
      {-3} | x_7^3             |
      {-3} | x_5^3             |
      {-3} | x_2x_3x_5         |
      {-3} | x_1x_4x_8         |
      {-3} | x_1x_4x_7         |
      {-3} | x_1x_4x_5         |
      {-4} | x_2x_4x_5x_7      |
      {-4} | x_2x_3x_4x_8      |
      {-4} | x_1x_8^2x_9       |
      {-4} | x_1x_7^2x_9       |
      {-4} | x_1x_7^2x_8       |
      {-4} | x_1x_5^2x_9       |
      {-4} | x_1x_5^2x_8       |
      {-4} | x_1x_5^2x_7       |
      {-5} | x_2x_5x_7^2x_9    |
      {-5} | x_2x_4x_7^2x_8    |
      {-5} | x_2x_4x_5^2x_8    |
      {-5} | x_2x_3x_8^2x_9    |
      {-6} | x_2x_5x_7x_8^2x_9 |
      {-6} | x_2x_5^2x_8^2x_9  |
      {-6} | x_2x_5^2x_7^2x_8  |
\end{verbatim}
\item[($\cI_{F_3}$)]
\begin{verbatim}
Leading Monomials for I_{F_3}
 	  {-2} | x_4x_9              |
      {-2} | x_4^2               |
      {-2} | x_3^2               |
      {-2} | x_2x_5              |
      {-2} | x_2^2               |
      {-2} | x_1x_3              |
      {-2} | x_1x_2              |
      {-2} | x_1^2               |
      {-3} | x_9^3               |
      {-3} | x_7^3               |
      {-3} | x_6^3               |
      {-3} | x_5^3               |
      {-3} | x_2x_3x_6           |
      {-3} | x_1x_4x_7           |
      {-3} | x_1x_4x_6           |
      {-3} | x_1x_4x_5           |
      {-4} | x_3x_4x_5x_7        |
      {-4} | x_2x_3x_4x_7        |
      {-4} | x_1x_7^2x_9         |
      {-4} | x_1x_6^2x_9         |
      {-4} | x_1x_6^2x_7         |
      {-4} | x_1x_5^2x_9         |
      {-4} | x_1x_5^2x_7         |
      {-4} | x_1x_5^2x_6         |
      {-5} | x_3x_5^2x_6x_9      |
      {-5} | x_3x_4x_5^2x_6      |
      {-5} | x_2x_4x_6^2x_7      |
      {-5} | x_2x_3x_7^2x_9      |
      {-6} | x_3x_5^2x_7^2x_9    |
      {-6} | x_3x_5^2x_6x_7^2    |
      {-6} | x_2x_6^2x_7^2x_9    |
      {-7} | x_3x_5x_6^2x_7^2x_9 |
\end{verbatim}
\item[($\cI_{F_4}$)]
\begin{verbatim}
Leading Monomials for I_{F_4}
	  {-2} | x_4^2               |
      {-2} | x_3x_7              |
      {-2} | x_3^2               |
      {-2} | x_2x_5              |
      {-2} | x_2^2               |
      {-2} | x_1x_3              |
      {-2} | x_1x_2              |
      {-2} | x_1^2               |
      {-3} | x_8^3               |
      {-3} | x_7^3               |
      {-3} | x_6^3               |
      {-3} | x_5^3               |
      {-3} | x_2x_3x_6           |
      {-3} | x_1x_4x_8           |
      {-3} | x_1x_4x_7           |
      {-3} | x_1x_4x_6           |
      {-3} | x_1x_4x_5           |
      {-4} | x_3x_4x_5x_8        |
      {-4} | x_2x_4x_6x_7        |
      {-4} | x_2x_3x_4x_8        |
      {-4} | x_1x_7^2x_8         |
      {-4} | x_1x_6^2x_8         |
      {-4} | x_1x_6^2x_7         |
      {-4} | x_1x_5^2x_8         |
      {-4} | x_1x_5^2x_7         |
      {-4} | x_1x_5^2x_6         |
      {-5} | x_4x_5^2x_6x_7      |
      {-5} | x_3x_4x_5^2x_6      |
      {-5} | x_2x_4x_7^2x_8      |
      {-5} | x_2x_4x_6^2x_8      |
      {-6} | x_4x_5^2x_7^2x_8    |
      {-6} | x_3x_5^2x_6x_8^2    |
      {-6} | x_2x_6^2x_7^2x_8    |
      {-7} | x_4x_5x_6^2x_7^2x_8 |
\end{verbatim}
\end{enumerate}
\vspace{0.5cm}
The code used to generate these terms is the following:
\begin{verbatim}
    clearAll
    --Number of Colors 
    k=3;
    
    -- Number of Vertices
    n=10;
    
    -- Edge set of the graph (uncomment desired ideal)
    E=matrix{{1,2},{1,3},{2,5},{2,6},{3,7},{3,8},{4,9},{4,10}}; --F1
    --E=matrix{{1,2},{1,3},{1,4},{2,5},{3,7},{3,8},{4,9},{4,10}}; --F2
    --E=matrix{{1,2},{1,3},{1,4},{2,5},{2,6},{3,7},{4,9},{4,10}}; --F3
    --E=matrix{{1,2},{1,3},{1,4},{2,5},{2,6},{3,7},{3,8},{4,9}}; --F4
    
    --Number of edges
    m=numgens target E;
    
    
    -- Define the ring and Monomial Order
    R=QQ[x_1..x_n, MonomialOrder=>GLex]
    
    --Polynomials p_u
    M_I=matrix{{x_1^k-1}};
    for i from 1 to n do (
        M_I=M_I|matrix{{x_i^k-1}};
        );
    
    --Polynomials q_uv
    num = (x_(E_(0,0))^k-x_(E_(0,1))^k);
    den = (x_(E_(0,0))-x_(E_(0,1)));
    M_J=matrix{{num//den}};
    
    for  j from 1 to m-1 do (
         num =(x_(E_(j,0))^k-x_(E_(j,1))^k);
         den = (x_(E_(j,0))-x_(E_(j,1)))
         mod = matrix{{num//den}};
         M_J=M_J|mod;
        );
    
    --Define the ideal
    ColId=ideal(M_I|M_J);
    
    --Find a Grobner basis
    G=gb(ColId);
    GensG=gens G;
    
    --Output the leading terms of the basis
    transpose(leadTerm(GensG))
\end{verbatim}

\section*{Acknowledgments}
The authors thank Jes\'us A. De Loera for valuable comments, insights and feedback during this research, in particular, for bringing \cite{LN17} to our attention. In addition, the authors thank Bertrand Guenin and Jochen Koenemann for their valuable comments on a preliminary version of the results. 

\bibliographystyle{siamplain}
\bibliography{references}
\end{document}